\documentclass[11pt,reqno]{amsart}
\usepackage{amsmath,amssymb,amsthm,amsrefs,mathrsfs}
\usepackage{enumerate}
\numberwithin{equation}{section}
\theoremstyle{definition}
\newtheorem{thm}{\sc Theorem}[section]

\newtheorem{lem}[thm]{\sc Lemma}
\newtheorem{rem}[thm]{\sc Remark}
\newtheorem{rems}[thm]{\sc Remarks}

\newtheorem{defi}[thm]{\sc Definition}

\newtheorem{exs}[thm]{\sc Examples}

\title[Wavelet transform]{Continuous Abstract Wavelet transform on Homogeneous Spaces}
\allowdisplaybreaks
\begin{document}
\begin{abstract}
The support of  wavelet transform associated with square integrable irreducible representation of a homogeneous space is shown to have infinite measure.  Pointwise homogeneous approximation property for wavelet transform has been investigated. An analogue of  Heisenberg type inequality has been also obtained for wavelet transform. 
\end{abstract}
\author[J. Sharma]{JYOTI SHARMA}
\address{Department of Mathematics, University of Delhi, Delhi, 110007, India.}
\email{jsharma3698@gmail.com}

\author[A. Kumar]{AJAY KUMAR$^\ast$}
\address{Department of Mathematics, University of Delhi, Delhi, 110007, India.}
\email[Corresponding author]{akumar@maths.du.ac.in}

\thanks{$^\ast$Corresponding author}

\maketitle

\section{Introduction}
\noindent Fourier transform is used to analyse the frequency property of a given signal. However, due to loss of information about time, other transforms like Gabor transform and wavelet transform have been found to be more useful. Continuous wavelet transform has been widely used in signal and image processing for investigating time-varying frequency.  Unlike Fourier analysis, wavelet analysis expands functions not in terms of trigonometric polynomials but in terms of wavelets, which are generated in the form of translations and dilations of a fixed function called the admissible wavelet. Wavelets obtained in this way have special scaling properties. They are localized in time and frequency, permitting a closer connection between the function being represented and their coefficients. These have been recently studied in harmonic analyis (see \cite{ali;00,Dav:96,fuhr;05,kam:12,Liu:09}).\\
 We  begin by defining wavelet transform for an arbitrary homogeneous space.
Let G be a locally compact group with left Haar measure $\mu$ and H a closed subgroup of $G$. The left coset space $G/H$ is a homogeneous space with quotient topology. Also, $G$ acts on $G/H$ via $(g,aH) \to  gaH$. We assume that $\Delta_G(h)= \Delta_H(h),$ for all $h\in H,$ where $\Delta_G$ and $\Delta_H$ are the modular function of $G$ and $H$ respectively.
For the pair $(G,H),$ a rho function is  a continuous function $\rho: G \to (0,\infty)$ such that $\rho(gh)= \rho(g)$ for all $g\in G$ and $h\in H$(see \cite[p. 65]{fol}).
 Let $\mathcal{U}(\mathcal{H}_{\pi})$ be the group consisting of all unitary operators on some Hilbert space $\mathcal{H}_{\pi}$. A continuous unitary representation of the homogeneous space $G/H$  is a map $\pi$ from $G/H$ into $\mathcal{U}(\mathcal{H}_{\pi})$ for which $gH \to \langle \pi(gH)\xi, \eta \rangle  $ is a continuous map from $G/H \to \mathbb{C}$ such that for each $\xi,\eta \in \mathcal{H}_{\pi}$ and  $g,k \in G$,
 \begin{align*}
 \pi(gkH) = \pi(gH)\pi(kH)\quad \text{ and } \quad  \pi(g^{-1}H) = \pi(gH)^*.
 \end{align*} 
 Let $\pi$ be a square integrable representation of $G/H$, i.e.  there exists some $\xi \in \mathcal{H}_{\pi}$ and $0< C_{\xi} < \infty$, satisfying 
 \begin{align}\label{eq1}
C_{\xi} =  \int_{G/H} \frac{\rho(e)}{\rho(g)}|\langle \xi, \pi(gH)\xi \rangle|^{2} d\mu_{G/H}(gH).
\end{align}
  Then, $\xi$ satisfying (\ref{eq1}) is called an admissible wavelet.
A continuous wavelet transform associated to $\xi$ is a linear operator $W_{\xi} : \mathcal{H}_{\pi} \to C(G/H) $ defined by 
 \begin{align*}
 (W_{\xi}\eta)(gH) = \left( \frac{\rho(e)}{\rho(g)}\right)^{1/2}\langle \eta, \pi(gH)\xi \rangle \text{ for all } \eta \in \mathcal{H}_{\pi}, gH \in G/H.
 \end{align*}  
\noindent If $\pi$ is irreducible, then $W_{\xi}$ is a bounded linear operator from Hilbert space $\mathcal{H}_{\pi} \text{ into } L^2(G/H)$ and $ L^2(G/H) \simeq L^2(G/H, d\mu_{G/H}),$ where $\mu_{G/H}$ is  a relatively $G$-invariant measure on $G/H$ which arises from the rho function $\rho.$ 
 Also, $W_{\xi}(\mathcal{H}_{\pi})$ is a reproducing kernel Hilbert space with pointwise bounded kernel and the operator ${C_{\xi}}^{-1/2}W_{\xi}$  is an isometry. For detailed study of wavelet transform on homogeneous space, one can refer to \cite{kam:12}.\\
In this paper, we first show that for  $\eta \in \mathcal{H}_{\pi}\setminus \{0\},$ the support of $W_{\xi}\eta$ is a set of infinite measure. We prove that every pair of admissible vectors possesses homogeneous approximation property. In addition, we study the wavelet groups of the form $B\ltimes A,$ where $B,A$ are locally  compact, type I groups. Moreover, lower estimate on norm of wavelet transform and Heisenberg type inequality have been obtained.
\section{Concentration of Wavelet transform}
\noindent Throughout this section, we assume that $H$ is compact and $\rho(g) = 1$ for all $g \in G$. Thus, $\mu_{G/H} $ is a $G$-invariant measure, i.e. $d\mu_{G/H}(saH) = d\mu_{G/H}(aH)$ for all $s\in G$ and $aH\in G/H$. To investigate the support of the wavelet transform $W_{\xi}\eta$, we begin with following lemmas.
\begin{lem}\label{cont}
If $ M_{0}\subseteq M \subseteq G/H$ are such that $ \mu_{G/H}(M) < \infty,$ then the function $h: G \to [0,\infty) $ such that $h(a) = \mu_{G/H}(M\cap aM_{0})$, is continuous.
\end{lem}
\begin{proof}
Let  $C_{00}(G/H)$ be the space of continuous functions with compact support. Since $\chi_{M_0} \in L^1(G/H)$ and $C_{00}(G/H)$  is dense in $L^{1}(G/H),$ therefore we  choose $f \in C_{00}(G/H)$ such that $\| \chi_{M_0} - f \|_{1} < \epsilon$, where $ \epsilon > 0$ is arbitrary. Then, we consider
\begin{flalign*}
|h(a)-h(b)| &\leq \int_{G/H} \left|\left(\chi_{aM_0} - {}_af\right)(xH)\right|d\mu_{G/H}(xH)&\\
&\qquad + \int_{M} \left|\left({}_af-{}_bf\right)(xH)\right|d\mu_{G/H}(xH)&\\
&\qquad +  \int_{G/H} \left|\left({}_bf -\chi_{bM_0}\right)(xH)\right|d\mu_{G/H}(xH)&\\
&= 2\|\chi_{M_0} - f \|_1 + \int_{M} \left|\left({}_af-  {}_bf\right)(xH)\right|d\mu_{G/H}(xH) &\\
& < 2\epsilon + \int_{M} \left|\left({}_af- {}_bf\right)(xH)\right|d\mu_{G/H}(xH).
\end{flalign*}

\noindent   The map $a \to  {}_af$ is  continuous  from $G \to (C_{00}(G/H), \|\cdot\|_{\infty})$ \cite[Proposition 1.25]{kan:13}, where ${}_af(xH) = f(a^{-1}xH)$ for all $a\in G$ and $xH \in G/H$. Hence, $h$ is a continuous function.
\end{proof}
\vspace{10pt}
\noindent In the next lemma, we generalize \cite[Lemma 2.1]{Hog} to homogeneous spaces.
\begin{lem}\label{conn}
Let $G_0$, the identity component of $G$ be non-compact and  $M_{0}\subseteq M \subset G/H$ be such that $ 0 <\mu_{G/H}(M_0)\leq \mu_{G/H}(M) < \infty.$ Then for $\epsilon > 0$ there exists $a \in G_0$ satisfying
\begin{align*}
\mu_{G/H}(M) < \mu_{G/H}(M\cup aM_{0}) < \mu_{G/H}(M)+\epsilon.
\end{align*}
\end{lem}
\begin{proof}
Define $h: G_0 \to \mathbb{R}^+$ by $h(a) = \mu_{G/H}(M\cup aM_{0}).$ Since $\mu_{G/H}$ is $G$-invariant, therefore $h(a) = \mu_{G/H}(M)+\mu_{G/H}(M_0) - \mu_{G/H}(M\cap aM_0).$ Then by  Lemma \ref{cont}, $h$ is a continuous function. Also $\mu_{G/H}$ is regular, so there exists a compact subset $K$ of $M$ such that $\mu_{G/H}(M\setminus K) < \mu_{G/H}(M_0)/2.$ Let $K = \{xH: x\in E \}$, where $E$ is a compact subset of $G$. Thus, $S = EHE^{-1}$ is a compact subset of $G$. Choose $a \in G_0\setminus S.$ For such a choice of $a,$ we have $ aK \cap K = \phi$. Moreover,
\begin{align*}
aM_0\cap K
&= a(M_0 \cap (K \cup K^{'}))\cap K\\
&= (aM_0 \cap aK \cap K)\cup (a(M_0 \cap K^{'})\cap K)\\
&= a(M_0 \cap K^{'})\cap K \subseteq a(M_0 \cap K^{'}). 
\end{align*} 
Therefore, it follows that
$$\mu_{G/H}(aM_0 \cap K)\leq \mu_{G/H}(a(M_0\cap K^{'}))=  \mu_{G/H}(M_0\setminus K) < \mu_{G/H}(M_0)/2.$$ 
Again, we have
\begin{flalign*}
 h(a)&= \mu_{G/H}(M\cup aM_0) &\\
&= \mu_{G/H}((M\cap K)\cup (M\cap K^{'})\cup (aM_0\cap K)\cup ( aM_{0}\cap K^{'}))&\\
&\geq \mu_{G/H}( (M\cap K)\cup ( aM_{0}\cap K^{'})) &\\
&= \mu_{G/H}(M)- \mu_{G/H}(M\cap K^{'})+ \mu_{G/H}(aM_0)- \mu_{G/H}(aM_0\cap K)&\\
&> \mu_{G/H}(M) = h(e). 
\end{flalign*} 
Thus, $h$ is a non-constant continuous function on the connected set $G_{0}$. So, we can choose $a \in G_{0}$ such that 
\begin{align*}
\mu_{G/H}(M) = h(e) < h(a) = \mu_{G/H}(M\cup aM_{0}) < \mu_{G/H}(M)+\epsilon = h(e)+ \epsilon.\hspace{1cm}\qedhere
\end{align*}
\end{proof}
\noindent In the next result, we prove that the support of wavelet transform $ W_{\xi}\eta $   has infinite measure. 

\begin{thm}\label{concth}
Let $G$ and $H$ be as in Lemma \ref{conn} and $\xi $ be an admissible wavelet associated with an irreducible square integrable representation $\pi$.  Then for any $\eta \in \mathcal{H}_{\pi} \setminus \{0\}$, the set $\{gH\in G/H: W_{\xi}\eta(gH)\neq 0 \}$ has infinite measure.  
\end{thm}
\begin{proof}
It is enough to show that for every subset $M$ of $ G/H$ with $\mu_{G/H}(M)< \infty $, the set $
\{ F \in W_{\xi}(\mathcal{H}_{\pi}): F = \chi_{M} F \} = \{0\}$.
For this assume that $M \subset G/H$ and $0\neq F_{0} \in W_{\xi}(\mathcal{H}_{\pi})$ with $F_0 = \chi_{M}F_0.$ Let $M_0 = \{xH \in G/H : F_{0}(xH)\neq 0 \}$. Then, 
 $$0 < \mu_{G/H}(M_0) \leq \mu_{G/H}(M) < \infty.$$ By Lemma \ref{conn}, there exists $a^{(1)} \in G_0 $ satisfying 
\[
\mu_{G/H}(M) < \mu_{G/H}(M \cup a^{(1)}M_0) < \mu_{G/H}(M) + \frac{\epsilon}{2}. 
\]
Take $M_1=M$ and $M_2 = M_1 \cup a^{(1)}M_0$ and again using  Lemma \ref{conn}, there exists $a^{(2)} \in G_0 $ satisfying 
\begin{align*}
\mu_{G/H}(M_2) < \mu_{G/H}(M_2 \cup a^{(2)}a^{(1)}M_0) < \mu_{G/H}(M_2) + \frac{\epsilon}{2^2}.
\end{align*}
 Continuing in the same manner, we get an increasing sequence $\{M_k\}_{k=1}^{\infty}$ of subsets of $G/H$ given by $$M_k = M_{k-1}\cup a^{(k-1)}\cdots a^{2}a^{1}M_0,$$  where  $a^{(i)} \in G_0$ for all $i = 1,2,\dots,k-1$ and satisfy
\begin{align}\label{lin-ind}
\mu_{G/H}(M_k)< \mu_{G/H}(M_k\cup a^{(k)}a^{(k-1)}\cdots a^{(1)}M_0)< \mu_{G/H}(M_k)+\frac{\epsilon}{2^k}.
\end{align}
Define $S = \bigcup\limits_{k=1}^{\infty} M_k.$ Then,
\begin{align*}
\mu_{G/H}(S) &= \lim\limits_{k \to \infty}\mu_{G/H}(M_k)
\leq \lim\limits_{k \to \infty } \left(\mu_{G/H}(M_{k-1})+ \frac{\epsilon}{2^{(k-1)}}\right)\\
&\leq  \mu_{G/H}(M)+ \lim\limits_{k \to \infty } \sum\limits_{i=1}^{k-1}\frac{\epsilon}{2^i}
=\mu_{G/H}(M)+ \epsilon < \infty.
\end{align*}
Consider the family $\{F_{k}\}_{k\in \mathbb{N}}$ of functions on $G/H$ as follows
\begin{flalign*}
&&&F_{1}(aH)= F_0(aH) &\\
&&&F_{k}(aH) = F_{k-1}((a^{k-1})^{-1}aH) \quad \text{for all } k \geq 2,\text{ and }   aH \in G/H.
\end{flalign*}
We will show by induction  that $F_{k}\in W_{\xi}(\mathcal{H}_{\pi})$ for all $k \in \mathbb{N}.$ Clearly, $F_1 = F_0 \in W_{\xi}(\mathcal{H}_{\pi}).$ Let us  assume that $F_{k-1} \in W_{\xi}(\mathcal{H}_{\pi}),\ \text{i.e.}\  F_{k-1} = W_{\xi}(\eta)$ for some $\eta \in \mathcal{H}_{\pi}.$ Then,
\begin{align*}
F_{k}(aH) &= W_{\xi}(\eta)((a^{(k-1)})^{-1}aH)\\
&=\left\langle \eta , \pi((a^{(k-1)})^{-1}aH)\xi \right\rangle\\
&= \left\langle \pi(a^{(k-1)}H)\eta, \pi(aH)\xi \right\rangle\\
&= W_{\xi}(\pi({a^{(k-1)}H)}\eta)
\end{align*}
which implies that $F_{k} \in W_{\xi}(\mathcal{H}_{\pi}).$ Using (\ref{lin-ind}), we conclude that the family $\{F_{k}\}_{k\in \mathbb{N}}$ of functions on $G/H$ is linearly independent vanishing outside the set $S$ of finite measure.
Since $W_{\xi}(\mathcal{H}_{\pi})$ is a reproducing kernel Hilbert space with point-wise bounded kernel, therefore each subspace of $W_{\xi}(\mathcal{H}_{\pi}) $ consisting of functions having support 
 on a set of finite measure must be of finite dimension \cite{Wilc:00}. If $W_{\xi}\eta \in W_{\xi}(\mathcal{H}_{\pi})$ such that $\mu_{G/H}\{xH : W_{\xi}\eta(xH)\neq 0 \}< \infty, $ then $W_{\xi}\eta = 0$ a.e. and hence $\eta = 0.$
Thus, the support of $W_{\xi}\eta$ is a set of infinite measure for every non-zero $\eta \in \mathcal{H}_{\pi}$.
 \end{proof}
\begin{rems} \noindent
\begin{enumerate}
 \item If we drop the irreducibility of $\pi$, then the conclusion of  Theorem \ref{concth} may not hold.
 Consider a locally compact  abelian topological group $G$ and the left regular representation $\pi$ of $G$. Then every non-zero function  $\psi \in C_{00}(G)$ is an admissible wavelet. Fix non-zero functions $f,\psi \in C_{00}(G).$
 Then, $W_{\psi}f(x)=\langle f, \pi(x)\psi \rangle = f* \psi^{\star}(x)$, where $*$ is the convolution and $\psi^{\star}(x)= \overline{\psi(-x)}$.  Hence $W_{\psi}f \in C_{00}(G).$

 \item If G is an abelian locally compact group and $\pi$ is an irreducible representation of $G$, then the support of wavelet transform is $G.$ In this case $W_{\xi}\eta(x) =  \overline{ \chi(x)}\langle \eta, \xi\rangle$, where $\chi$ is the character of $G$ associated with $\pi.$ 
This may not be true if we consider a non-abelian compact group G.
 Consider $G = \mathbb{T}\times D_3, \text{ where } D_{3} = \{x,y : x^3 = e = y^2, xy = yx^{-1} \}.$
 Let $\pi$ be an irreducible unitary representation of G given by:
 \begin{align*}
 \pi(t,x) = e^{int} \otimes \begin{bmatrix}
 \omega & 0\\
 0 & \omega^{-1}
 \end{bmatrix},\ 
 \pi(t,y) = e^{int} \otimes \begin{bmatrix}
 0 & 1\\
 1 & 0
 \end{bmatrix}, 
 \end{align*}
  where $\omega^{3} = 1$,  $ n\in \mathbb{Z}$ is fixed and $\otimes $ being the outer tensor product.
 If we consider $\xi =  \eta =  1\otimes e_1 \in \mathbb{C} \otimes \mathbb{C}^2$, then $W_{\xi}\eta(t,y)= 0.$
 Thus, $\{ (s,u)\in \mathbb{T}\times D_{3}: W_{\xi}\eta (s,u)\neq 0\} \subseteq \mathbb{T} \times D^{3}/{\{y\}}$ which implies that
 \begin{align*}
\mu_{G}\{ (s,u)\in \mathbb{T}\times D_{3}: W_{\xi}\eta(s,u)\neq 0\} \leq \mu_{\mathbb{T}}(\mathbb{T})\frac{\#(D^{3}/{\{y\}})}{\#(D_3)} < 1
\end{align*} where $\#(D_{3})$ is the cardinality of $D_3.$ 
 \end{enumerate}
 \end{rems}
\noindent We conclude this section section by giving  examples of group having  a square integrable irreducible representation.
\begin{exs}\label{eg}
1. Groups of the type $ H \ltimes \mathbb{R}^n ,$ where $H$ is a closed subgroup of $GL_n(R).$
\begin{enumerate}[(i)]

\item Let $G =  A  \ltimes \mathbb{R}^2$, where 
$A = \left \{ \begin{pmatrix}
a_1 & 0\\
0 & a_2 
\end{pmatrix} : a_1,a_2 \in \mathbb{R}^* \right\}$.\\
 A  square integrable irreducible representation $\pi$ of $G$ on $L^2(R^2)$ is given by
 $$\pi\left[ \begin{pmatrix}
a_1 & 0\\
0 & a_2
\end{pmatrix}, \begin{pmatrix}
 x_1\\
 x_2
\end{pmatrix}\right] g \begin{pmatrix}
y_1\\
y_2
\end{pmatrix} = \frac{1}{\sqrt{|a_1a_2|}}g\begin{pmatrix}
(y_1-x_1)/a_1\\ (y_2 -x_2)/a_2
\end{pmatrix}.$$
Also, $\psi \in L^2(\mathbb{R}^2)$ is admissible if and only if 
$ C_{\psi} = \int_{\mathbb{R}^2} \Big| \frac{\widehat{\psi}(\gamma_1,\gamma_2)}{\sqrt{|\gamma_1\gamma_2|}}\Big|^2 d\gamma_1 d\gamma_2$ $ < \infty$  \cite{Dav:96}.

\item $G =A \ltimes \mathbb{R}^2 ,$ where $A = \left\{\begin{pmatrix}
a & b\\
0 & 1
\end{pmatrix}:a\in \mathbb{R}^*,b\in \mathbb{R}\right\}$.\\
 A square integrable irreducible representation $\pi$ on $L^2(\mathbb{R}^2)$ is given by  
$$\pi\left[ \begin{pmatrix}
a & b\\
0 & 1
\end{pmatrix}, \begin{pmatrix}
 x\\
 y
\end{pmatrix} \right] g(u,v) = a^{-1}g\left(
a^{-1}(u-x+b(v-y)), v-y\right).$$
Also, $\psi \in L^2(\mathbb{R}^2)$ is admissible if and only if  $ C_{\psi}=\int_{\mathbb{R}^2}\frac{|\widehat{\psi}(\xi_1,\xi_2)|^2}{|\xi_1|^2}d\xi_1d\xi_2 < \infty.$
\item  Shear group $\mathcal{S} = (\mathbb{R}^{*} \ltimes \mathbb{R}^{n-1})\ltimes \mathbb{R}^{n} \text{ for } n\geq 2 $ with group operation
$(a,s,t).(a',s',t') = (aa',s+|a|^{1-1/n}s^{'}, t+S_s A_a t'),$ where
 $$
 A_a = \begin{pmatrix}
 a & 0_{n-1}^T \\
 0_{n-1} & sgn(a)|a|^{1/n}I_{n-1} 
\end{pmatrix}  \text{ and }
S_s = \begin{pmatrix}
1 & s^{T}\\
0_{n-1} & I_{n-1}
\end{pmatrix}.$$
A square integrable irreducible representation $\pi$ on $L^2{(\mathbb{R}^n)}$ is given by 
$$\pi(a,s,t)f(x) = |a|^{\frac{1}{2n}-1}f(A_a^{-1}S_s^{-1}(x-t))$$ and function $\psi \in L^2(\mathbb{R}^2) $ is admissible if and only if $C_{\psi} = \int_{\mathbb{R}^n} \frac{|\widehat{\psi}(\omega)|^2}{|\omega_1|^n} d\omega < \infty$ \cite{dah;10}.
\end{enumerate}
2. Low dimensional nilpotent Lie groups.
\begin{enumerate}[(i)]
\item Reduced Weyl Heisenberg group, $ \mathbb{R}^{n} \ltimes (\mathbb{R}^{n} \times \mathbb{T})\text{ for } n\in \mathbb{N}$ with the group operation $(a_1, b_1,t_1)(a_2,b_2,t_2) = (a_1+a_2, b_1+b_2, t_1 t_2e^{2\pi i a_1\cdot b_2}).$
An irreducible square integrable representation $\sigma$ on the Hilbert space $ L^2(\mathbb{R}^n)$ is given by 
$$\sigma(a,b,t)f(x)= te^{2\pi i b(x-a)}f(x-a).$$
Every $\psi \in L^2(R^n)$ is an admissible wavelet and $C_{\psi} = \|\psi\|_2^2$.
\end{enumerate}

\noindent For the group structure of 5-dimensional groups $G_{5,1}, G_{5,3},G_{5,6}$ given below (see \cite{Niel:83}).
\begin{enumerate}[(i)]
\item [(ii)] $G_{5,1}/\mathbb{Z},$ an irreducible square integrable  representation $ \sigma $ on the  Hilbert space $ L^2(\mathbb{R}^2)$ is given by 
$$\sigma(t,x_2\dots , x_5)f(s_1,s_2) = te^{-2\pi i(x_2s_1+x_4s_2)}f(s_1-x_3,s_2-x_5),$$
where $t\in \mathbb{T}, x_i \in \mathbb{R}.$ 
Every $\psi \in L^2(R^2)$ is an admissible wavelet and $C_{\psi} = \|\psi\|_2^2$.
\item [(iii)] $G_{5,3}/\mathbb{Z},$ a square integrable irreducible representation $ \sigma $ on the Hilbert space $ L^2(\mathbb{R}^2)$ is given by 
$$ \text{ } \qquad \sigma(t,x_2\dots , x_5)f(s_1,s_2) = te^{-2\pi i(x_3x_4 -x_5s_1+x_2s_2-\frac{1}{2}x_4s_2^2)}f(s_1-x_3,s_2-x_5),$$
where $t\in \mathbb{T}, x_i \in \mathbb{R}.$ 
Every $\psi \in L^2(R^2)$ is an admissible wavelet and $C_{\psi} = \|\psi\|_2^2$.
\item[(iv)]  $G_{5,6}/\mathbb{Z},$ an irreducible square integrable representation $ \sigma $ on the Hilbert space $ L^2(\mathbb{R}^2)$ is given by 
\begin{align*} 
\quad \sigma(t,x_2\dots , x_5)f(s_1,s_2) \\
 = t\exp2\pi i&\left( -\frac{1}{2}x_4^2x_5-\frac{1}{6}x_4x_5^3\right.-x_3s_1-x_4x_5+\frac{1}{6}x_5^3s_1\\
 &-\frac{1}{2}x_5s_1^2-x_2s_2+\frac{1}{2}x_4x_5^2s_2+\frac{1}{2}x_5^3s_2^2\\
 & -\frac{1}{2}x_4x_5\left.-\frac{1}{2}x_5^2s_1s_2+\frac{1}{2x_5s_1s_2^2} \right)f(s_1-x_4,s_2-x_5),
\end{align*}
where $t\in \mathbb{T}, x_i \in \mathbb{R}.$ 
Every $\psi \in L^2(R^2)$ is an admissible wavelet and $C_{\psi} = \|\psi\|_2^2$.
\end{enumerate}
3. Let $A$ be a locally compact group  with non-compact identity component having a square integrable irreducible representation $\sigma$ and $B$ be a compact group of the form $H\rtimes K$. Let $\tau$ be an irreducible representation of $B$ which is identity on $K$. Then, $ G = (A \times B)/({\{0\}\times K})$ has an irreducible square integrable representation $\pi = \sigma \otimes \tau$, where $\otimes $ being the outer tensor product \cite[Theorem 7.17]{fol}. 
\end{exs}

\section{Homogeneous Approximation Property}
\noindent In this section, we deal with homogeneous approximation property for the wavelet transform on homogeneous space $G/H,$ where $H$ is a closed subgroup of $G$. For an  irreducible unitary representation $\pi$ of $G/H$ with representation space $\mathcal{H}_{\pi}$, 
a pair $ (\xi_1, \xi_2) $ consisting of admissible wavelets in $\mathcal{H}_{\pi}$ is said to be an admissible pair if
$\langle \xi_1, \xi_2 \rangle \neq 0$ and $$C_{\xi_1
\xi_2} = \frac{1}{\langle \xi_1,\xi_2 \rangle}\int_{G/H}W_{\xi_1}(\xi_1)(gH)\ \overline{W_{\xi_2}(\xi_2)(gH)}\ d\mu_{G/H}(gH) $$ is a non-zero scalar.
\noindent As proved in \cite[Theorem 2.1]{kam:12}, for admissible wavelets $\xi_1,\xi_2$,
 the operator  $ W_{\xi_2}^*W_{\xi_1}$ intertwines with $\pi$. Thus,   there exists a scalar $c$ such that $W_{\xi_2}^* W_{\xi_1} = cI$, where $I$ is the identity operator. Therefore, we have
\begin{align*}
c\langle \eta_1,\eta_2 \rangle
&= \langle W_{\xi_1}\eta_1,W_{\xi_2}\eta_2 \rangle\\
&=\int_{G/H}W_{\xi_1}\eta_1(gH)\ \overline{W_{\xi_2}\eta_2(gH)}\ d\mu(gH),
\end{align*}
for all $\eta_1,\eta_2 \in \mathcal{H}_{\pi} $. Taking $ \eta_1=\xi_1,  \eta_2=\xi_2$, 
 we have $$c\langle \xi_1,\xi_2 \rangle = \int_{G/H}W_{\xi_1}\xi_1(gH)\ \overline{W_{\xi_2}\xi_2(gH)}\ d\mu(gH).$$
If $(\xi_1,\xi_2)$ is an admissible pair, then 
$\langle \eta_1,\eta_2 \rangle= C_{\xi_1\xi_2}^{-1}\langle W_{\xi_1}\eta_1, W_{\xi_2}\eta_2 \rangle$ which further implies that every 
 $\eta\in \mathcal{H}_{\pi}$ can be decomposed as 
\begin{align*}
\eta &= C_{\xi_1\xi_2}^{-1}\int_{G/H}\left(\frac{\rho(e)}{\rho(g)}\right)^{1/2}W_{\xi_1}\eta(gH)\ \pi(gH)\xi_2\ d\mu(gH),
\end{align*}
where  the integral is defined in the weak sense.
\begin{defi}
An admissible pair $(\xi_1,\xi_2)$ is said to have homogeneous approximation property in $\mathcal{H}_{\pi}$, if for any $\eta \in \mathcal{H}_{\pi}$ and $\epsilon > 0,$ there exists some compact subset $K$ of $G/H$ such that every compact subset $K^{'}$ of $G/H$ with $K \subseteq K^{'}$ satisfies
$$\left\|\pi(sH)\eta - C_{\xi_1\xi_2}^{-1}\int_{s\cdot K^{'}}\left(\frac{\rho(e)}{\rho(a)}\right)^{1/2}W_{\xi_1}\pi(sH)\eta(aH)\ \pi(aH)\xi_2\ d\mu_{G/H}(aH)\right\|\leq \epsilon,$$
for all $s \in G.$
\end{defi}
\begin{thm}
Let $G/H$ be a $\sigma $-compact homogeneous space with  an irreducible square integrable representation $\pi$. Then every admissible pair $(\xi_1,\xi_2)$ possesses homogeneous approximation property in $\mathcal{H}_{\pi}$.  
\end{thm}
\begin{proof}
Let $K$ be an arbitrary compact subset of $G/H$ and $K^{'}$ be a compact subset of $G/H$ such that $K\subseteq K^{'}$. Then for $\eta \in \mathcal{H}_{\pi}$ and $sH \in G/H,$ we have
\begin{align*}
&\left\|\pi(sH)\eta - C_{\xi_1\xi_2}^{-1}\int_{s\cdot K^{'}}\left(\frac{\rho(e)}{\rho(a)}\right)\langle \pi(sH)\eta, \pi(aH)\xi_1 \rangle\ \pi(aH)\xi_2\  d\mu_{G/H}(aH)\right\|^2\\
&=\sup\limits_{\|\zeta\|=1}\left|\left(\left\langle  \pi(sH)\eta - C_{\xi_1\xi_2}^{-1}\int_{s\cdot K^{'}}\left(\frac{\rho(e)}{\rho(a)}\right)\langle \pi(s)\eta,
\pi(aH)\xi_1 \rangle\right.\right. \right.\\
&\qquad\qquad \qquad \qquad \left.\times \pi(aH)\xi_2\ d\mu_{G/H}(aH)\right) , \left.\zeta\rangle  \right|^2\\
&=\sup\limits_{\|\zeta\|=1}\left|  C_{\xi_1\xi_2}^{-1} \int_{aH \notin s\cdot K^{'}}\left(\frac{\rho(e)}{\rho(a)}\right)\langle \pi(sH)\eta, \pi(aH)\xi_1 \rangle\right. \\
&\qquad \qquad \qquad \qquad \times \left. \left\langle\pi(aH)\xi_2,\zeta \right\rangle d\mu_{G/H}(aH) \right|^{2}\\
&\leq C_{\xi_1\xi_2}^{-2}\int_{aH \notin s\cdot K^{'}}\left(\frac{\rho(e)}{\rho(a)}\right)\big|\langle \pi(sH)\eta, \pi(aH)\xi_1 \rangle\big|^{2}\ d\mu_{G/H}(aH)\\
&\qquad \qquad \qquad \qquad\times \sup\limits_{\|\zeta\| = 1}\int_{G/H}\left(\frac{\rho(e)}{\rho(a)}\right)\big|\langle\pi(aH)\xi_2,\zeta \rangle\big|^2\ d\mu_{G/H}(aH)\\
&\leq C_{\xi_1\xi_2}^{-2}C_{\xi_2}\int_{aH \notin s\cdot K}\left(\frac{\rho(e)}{\rho(a)}\right)\big|\langle \pi(sH)\eta, \pi(aH)\xi_1 \rangle\big|^{2}\ d\mu_{G/H}(aH)\\
&\leq C_{\xi_1\xi_2}^{-2}C_{\xi_2}\int_{aH \notin  K}\left(\frac{\rho(e)}{\rho(sa)}\right)\big|\langle \eta, \pi(aH)\xi_1 \rangle\big|^{2}\ d\mu_{G/H}(saH)\\
&= C_{\xi_1\xi_2}^{-2}C_{\xi_2}\int_{aH \notin K}\left(\frac{\rho(e)}{\rho(a)}\right)\big|\langle \eta, \pi(aH)\xi_1 \rangle\big|^{2}\ d\mu_{G/H}(aH).
\end{align*}
But, $G/H$ is $\sigma$-compact, therefore, there exists a sequence  $\{K_n\}_{n\in \mathbb{N}}$ of compact subsets of $G/H$ such that $K_n\subset K_{n+1}$  for all $n\in \mathbb{N}$ and   $G/H = \bigcup\limits_{n=1}^{\infty}K_n$. Since $\left|\chi_{K_n}W_{\xi_1}\eta - W_{\xi_1}\eta\right|^2  \to 0 $, as $n\to \infty$ and $\left|\chi_{K_n}W_{\xi_1}\eta - W_{\xi_1}\eta\right|^2 \leq 4|W_{\xi_1}\eta|^2$,
therefore by Lebesgue dominated convergence theorem, it follows that
\begin{flalign*}
&\int_{G/H} \left|\chi_{K_n}W_{\xi_1}\eta(aH) - W_{\xi_1}\eta(aH)\right|^2\ d\mu_{G/H}(aH)\to 0, \ \text{ as } n\to \infty. & 
\end{flalign*}
 Thus, there exists a $K_n$ such that 
\begin{flalign*}
&\int_{G/H} \left|\chi_{K_n}W_{\xi_1}\eta(aH) - W_{\xi_1}\eta(aH)\right|^2 d\mu_{G/H}(aH)\leq \frac{\epsilon}{ C_{\xi_1\xi_2}^{-2}C_{\xi_1}},&
\end{flalign*}
or, equivalently,
\begin{flalign*}
&\int_{aH\notin K_n} \left(\frac{\rho(e)}{\rho(a)}\right)|\langle \eta, \pi(aH)\xi_1 \rangle|^{2}\ d\mu_{G/H}(aH)\leq \frac{\epsilon}{ C_{\xi_1\xi_2}^{-2}C_{\xi_1}}.&
\end{flalign*} 
Therefore, we can choose  a compact subset  $K$ of $G/H$ such that
\begin{align*} \left\|\pi(sH)\eta - C_{\xi_1\xi_2}^{-1}\int_{s\cdot K^{'}}\left(\frac{\rho(e)}{\rho(a)}\right)\langle \pi(sH)\eta, \pi(aH)\xi_1 \rangle \pi(aH)\xi_2\ d\mu_{G/H}(aH)\right\|\leq \epsilon\end{align*}
 for every compact subset $K^{'}$ of $G/H$ containing $K$ and every $s\in G.$ 
\end{proof}
\noindent We now  consider similitude group $ G=(\mathbb{R}^+\times S0(n))\ltimes \mathbb{R}^n,$ with group operation 
$$(a,r_{\theta},t)(a_1,{r_{\theta}}_1,t_1) = (aa_1,r_{\theta}{r_{\theta}}_1, t+ ar_{\theta}t_1 ). $$ 
The left Haar measure on $G$ is given by $ \frac{dad\theta dt}{a^{n+1}}.$ 
The mapping $\pi: G \to \mathcal{U}(L^2(\mathbb{R}^n))$ defined by
$$\pi(a,r_{\theta},t)f(x) = a^{-n/2} \psi\left(r_{\theta}^{-1}\left(\frac{x-t}{a}\right)\right)$$ is a square integrable irreducible representation of $G$. For $\psi\in L^2(\mathbb{R}^n)$, $$C_{\psi} = \int_{\mathbb{R}^{+}}\int_{S0(n)}\frac{\widehat{|\psi}(ar_\theta^{-1}(\omega))|^2}{a}\ da\ d\theta,$$ is  independent of almost every  $\omega \in \mathbb{R}^n$ (see \cite{ali;00}). 
Also, for an admissible pair $(\psi_1,\psi_2)$,
$$C_{\psi_1\psi_2} = \int_{\mathbb{R}^{+}}\int_{S0(n)}\frac{\overline{\widehat{\psi_1}(ar_\theta^{-1}(\omega))}\ \widehat{\psi_2}(ar_\theta^{-1}(\omega))}{a}\ da\ d\theta$$ 
is independent of $\omega$ upto a measure zero set.\\
\noindent In the next result, we further deal with the point-wise homogeneous approximation property for wavelet transform on similitude group. The idea of  results is motivated by \cite{Liu:09} and \cite{yu;16}.   
 \begin{thm}
Let $(\psi_1,\psi_2)$ be an admissible pair in $L^2(\mathbb{R}^n)$. For $f\in L^2(\mathbb{R}^n)$ and $A_2 > A_1 > 0,$ define
\begin{align*}
f_{A_1A_2}(x)= C_{\psi_1\psi_2}^{-1}\int_{A_1}^{A_2}\int_{S0(n)}\int_{\mathbb{R}^n} \langle f, \pi(a,r_{\theta},t)\psi_1\rangle\ \pi(a,r_{\theta},t)\psi_2(x)\ \frac{dtd\theta da}{a^{n+1}}.
\end{align*}
Then, $f_{A_1A_2} \in L^2(\mathbb{R}^n)$ and
\begin{align*}
\widehat{{f_{A_1A_2}}}(\omega) = C_{\psi_1\psi_2}^{-1}\widehat{f}(\omega)\int_{A_1}^{A_2}\int_{S0(n)} \widehat{\psi_{1}}(ar_\theta^{-1}(\omega))\ \widehat{\psi_2}(ar_\theta^{-1}(\omega))\ \frac{d\theta da}{a}.
\end{align*}
\end{thm}
\begin{proof}
First we show that $f_{A_1A_2}$ is well defined. For this, we compute
\begin{align*}
&\int_{A_1}^{A_2}\int_{S0(n)}\int_{\mathbb{R}^n} \big|\langle f, \pi(a,r_{\theta},t)\psi_1\rangle\ \pi(a,r_{\theta},t)\psi_2(x)\big|\ \frac{dtd\theta da}{a^{n+1}}\\
&\leq\int_{A_1}^{A_2}\int_{S0(n)}\left( \int_{\mathbb{R}^n} \big|\langle f, \pi(a,r_{\theta},t)\psi_1\rangle \big|^2 dt\right)^{\frac{1}{2}}\left( \int_{\mathbb{R}^n}\big|\pi(a,r_{\theta},t)\psi_2(x)\big|^2dt \right)^{\frac{1}{2}}\frac{d\theta da}{a^{n+1}}\\
&\leq\int_{A_1}^{A_2}\int_{S0(n)}\left( \int_{\mathbb{R}^n} \big|\langle f, \pi(a,r_{\theta},t)\psi_1\rangle \big|^2 dt\right)^{\frac{1}{2}}\|\psi_2\|_{2}\ \frac{d\theta da}{a^{n+1}}\\
&\leq \|\psi_2\|_{2}\left(\int_{A_1}^{A_2}\int_{S0(n)}\int_{\mathbb{R}^n} \big|\langle f, \pi(a,r_{\theta},t)\psi_1\rangle \big|^2 \frac{ dtd\theta da}{a^{n+1}}\right)^{1/2}\left(\int_{A_1}^{A_2}\int_{S0(n)}\frac{d\theta da}{a^{n+1}} \right)^{1/2}\\
&=\|\psi_2\|_{2}\sqrt{C_{\psi_1}}\|f\|_2\ n^{-\frac{1}{2}}(A_1^{-n}-A_2^{-n})< \infty.
\end{align*}
Let $g\in  L^1\cap L^2(\mathbb{R}^n)$ be arbitrary. Then 
\begin{align*}
&\int_{\mathbb{R}^n} \left|g(x)f_{A_1A_2}(x)\right|dx \\
&\leq \int_{\mathbb{R}^n}\left(|g(x)| \int_{A_1}^{A_2}\int_{S0(n)}\int_{\mathbb{R}^n} |\langle f, \pi(a,r_{\theta},t)\psi_1\rangle \pi(a,r_{\theta},t)\psi_2(x)|\frac{dtd\theta da}{a^{n+1}}    \right)dx\\
& < \infty.
\end{align*}
Thus, using Tonneli theorem, the function  $g(\cdot)\langle f, \pi(\cdot,\cdot,\cdot)\psi_1\rangle \pi(\cdot,\cdot,\cdot)\psi_2(\cdot) $ is integrable.
Again, we have
\begin{align*}
& \int_{\mathbb{R}^n} f_{A_1A_2}(x)\overline{g(x)}\ dx \\
 &= C_{\psi_1\psi_2}^{-1} \int_{\mathbb{R}^n}\overline{g(x)}\left(\int_{A_1}^{A_2}\int_{S0(n)}\int_{\mathbb{R}^n} \langle f, \pi(a,r_{\theta},t)\psi_1\rangle \pi(a,r_{\theta},t)\psi_2(x)\ \frac{dtd\theta da}{a^{n+1}}\right)dx\\
 &=C_{\psi_1\psi_2}^{-1}\int_{A_1}^{A_2}\int_{S0(n)}\int_{\mathbb{R}^n} \langle f, \pi(a,r_{\theta},t)\psi_1\rangle\left(\int_{\mathbb{R}^n}\overline{g(x)}\pi(a,r_{\theta},t)\psi_2(x)dx\right)\ \frac{dtd\theta da}{a^{n+1}}\\
 &=C_{\psi_1\psi_2}^{-1}\int_{A_1}^{A_2}\int_{S0(n)}\int_{\mathbb{R}^n} \langle f, \pi(a,r_{\theta},t)\psi_1\rangle\langle \pi(a,r_{\theta},t)\psi_2, g\rangle\ \frac{dtd\theta da}{a^{n+1}}.
 \end{align*}
 Thus, 
 \begin{flalign*}
 \left|\int_{\mathbb{R}^n} f_{A_1A_2}(x)\overline{g(x)}\ dx\right|
& \leq C_{\psi_1\psi_2}^{-1}\int_{\mathbb{R}^+}\int_{S0(n)}\int_{\mathbb{R}^n}| \langle f, \pi(a,r_{\theta},t)\psi_1\rangle\langle \pi(a,r_{\theta},t)\psi_2, g\rangle|\ \frac{dtd\theta da}{a^{n+1}}&\\
&\leq  C_{\psi_1\psi_2}^{-1}\|f\|_2C_{\psi_1}C_{\psi_2}\|g\|_2.
 \end{flalign*}
Therefore, by  Riesz representation theorem $f_{A_1A_2}\in L^2(\mathbb{R}^n).$ 
Now on using Plancherel Theorem, we have
\begin{align*}
&&&\langle f_{A_1A_2}, g\rangle &\\
 &&&= C_{\psi_1\psi_2}^{-1}\int_{A_1}^{A_2}\int_{S0(n)}\int_{\mathbb{R}^n} \langle f, \pi(a,r_{\theta},t)\psi_1\rangle\ \langle \pi(a,r_{\theta},t)\psi_2, g\rangle\ \frac{dtd\theta da}{a^{n+1}} & \\
&&&= C_{\psi_1\psi_2}^{-1}\int_{A_1}^{A_2}\int_{S0(n)}\int_{\mathbb{R}^n} (f*\pi(a,r_\theta,0)\psi_1^{\star}(t))\ (g*\pi(a,r_\theta,0)\psi_2^{\star}(t))\ \frac{dtd\theta da}{a^{n+1}} &\\
&&&= C_{\psi_1\psi_2}^{-1}\int_{A_1}^{A_2}\int_{S0(n)}\int_{\mathbb{R}^n}\widehat{f}(\omega)\ \widehat{\pi(a,r_\theta,0)\ \psi_1^{\star}}(\omega)\ \overline{\widehat{g}(\omega)}\ \overline{\widehat{\pi(a,r_\theta,0)\psi_2^{\star}}}(\omega)\ \frac{d\omega d\theta da}{a^{n+1}} &\\
&&&= C_{\psi_1\psi_2}^{-1}\int_{\mathbb{R}^n}\widehat{f}(\omega) \left( \int_{A_1}^{A_2}\int_{S0(n)}\overline{\widehat{\psi_1}(ar_\theta^{-1}(\omega))}\ \widehat{\psi_2}(ar_\theta^{-1}(\omega))\ \frac{da d\theta}{a} \right) \overline{\widehat{g}(\omega)} d\omega.
\end{align*}
Since $g \in L^{1}\cap L^2(\mathbb{R}^n)$, therefore we have
\begin{align*}
\widehat{f_{A_1A_2}}(\omega) = C_{\psi_1\psi_2}^{-1}\widehat{f}(\omega)\int_{A_1}^{A_2}\int_{S0(n)} \widehat{\psi_{1}}(ar_\theta^{-1}(\omega))\ \widehat{\psi_2}(ar_\theta^{-1}(\omega))\frac{dad\theta}{a}.\qquad\qquad \qquad \qquad\qedhere
\end{align*}
 \end{proof}
\begin{thm}
Let $(\psi_1,\psi_2)$ be an admissible pair in $L^2(\mathbb{R}^n)$ and  $f\in L^2(\mathbb{R}^n)$ with $\widehat{f} \in L^1(\mathbb{R}^n)$. Then for any $\epsilon,p > 0$, there exist $A_2 > A_1 > 0 $ such that for any $(a,r_{\theta},t) \in G$ with $0< p\leq a$ and $0 < A_1^{'}\leq A_1$, $A_2^{'} \geq A_2,$  the following holds
\begin{align*}
&\left\|\pi(a,r_{\theta},t)f(x)-\left(C_{\psi_1\psi_2}^{-1}\int_{A_1^{'}a}^{A_2^{'}a}\int_{S0(n)}\int_{\mathbb{R}^n}\langle \pi(a,r_{\theta},t)f,\pi(a_1,{r_{\theta}}_1,t_1)\psi_1\rangle\right.\right.\\
&\qquad\qquad \qquad \qquad \qquad \qquad \left.\left.\times \pi(a_1,{r_{\theta}}_1,t_1)\psi_2(x)\ \frac{dt_1d\theta_1da_1}{a_1^{n+1}} \right)\right\|_{\infty}\leq \epsilon.
\end{align*}
\end{thm}
\begin{proof}
For arbitrary $0 < A_1^{'}< A_{2}^{'}$ and $(a,r_{\theta},t)\in G$ with $p\leq a$, define
\begin{flalign*}
&f_{A_1^{'}a}(x)=C_{\psi_1\psi_2}^{-1} \int_{0}^{A_1^{'}a}\int_{S0(n)}\int_{\mathbb{R}^n}\langle \pi(a,r_{\theta},t)f,\pi(a_1,{r_{\theta_1}},t_1)\psi_1\rangle &\\
&  \qquad \qquad \qquad \qquad \qquad \qquad \qquad \times \pi(a_1,{r_{\theta}}_1,t_1)\psi_2(x)\ \frac{dt_1d\theta_1da_1}{a_1^{n+1}}
\end{flalign*}
and
 \begin{flalign*}
&f_{A_2^{'}a}(x)=C_{\psi_1\psi_2}^{-1} \int_{A_2^{'}a}^{\infty}\int_{S0(n)}\int_{\mathbb{R}^n}\langle \pi(a,r_{\theta},t)f,\pi(a_1,{r_{\theta_1}},t_1)\psi_1\rangle &\\
& \qquad \qquad \qquad \qquad \qquad \qquad \qquad \times\pi(a_1,{r_{\theta}}_1,t_1)\psi_2(x)\ \frac{dt_1d\theta_1da_1}{a_1^{n+1}}.
\end{flalign*}
Then,
\begin{align*}
&\left\|\pi(a,r_{\theta},t)f(x)-\left(C_{\psi_1\psi_2}^{-1}\int_{A_1^{'}a}^{A_2^{'}a}\int_{S0(n)}\int_{\mathbb{R}^n}\left(\langle \pi(a,r_{\theta},t)f,\pi(a_1,{r_{\theta}}_1,t_1)\psi_1\rangle\right)\right.\right.\\
&\qquad \qquad \qquad \qquad  \qquad \left.\left.\times\  \left(\pi(a_1,{r_{\theta}}_1,t_1)\psi_2(x)\right)\ \frac{dt_1d\theta_1da_1}{a_1^{n+1}}\right) \right\|_{\infty}\\
& = \Big\| f_{A_1^{'}a}(x)+ f_{A_2^{'}a}(x)  \|_{\infty} \leq  \|\widehat{f_{A_1^{'}a}}\|_{1} +  \|\widehat{f_{A_2^{'}a}}\Big\|_{1}\\
&\leq|C_{\psi_1\psi_2}^{-1}|
\int_{\mathbb{R}^n}\left|\widehat{\pi(a,r_{\theta},t)f}(\omega)\right|\int_{S0(n)}\int_{0}^{A_1^{'}a} |\widehat{\psi_{1}}(a_1 {r_{\theta}}_1^{-1}(\omega))\ \widehat{\psi_2}(a_1 {r_{\theta}}_1^{-1}(\omega))| \ \frac{da_1d\theta_1}{a_1}d\omega \\
&\qquad+ |C_{\psi_1\psi_2}^{-1}|\int_{\mathbb{R}^n}\left|\widehat{\pi(a,r_{\theta},t)f}(\omega)\right|\int_{S0(n)}\int_{A_2^{'}a}^{\infty} |\widehat{\psi_{1}}(a_1{r_\theta}_1^{-1}(\omega))\ \widehat{\psi_2}(a_1 {r_{\theta}}_1^{-1}(\omega))|\ \frac{da_1d\theta_1}{a_1}d\omega. \\
 &\leq |C_{\psi_1\psi_2}^{-1}|\int_{\mathbb{R}^n}|a|^{n/2}|\widehat{f}(ar_{\theta}^{-1}\omega)|\int_{S0(n)}\int_{0}^{A_1^{'}a} |\widehat{\psi_{1}}(a_1 {r_{\theta}}_1^{-1}(\omega))\ \widehat{\psi_2}(a_1 {r_{\theta}}_1^{-1}(\omega))|\ \frac{da_1d\theta_1}{a_1}d\omega \\
&\qquad+ |C_{\psi_1\psi_2}^{-1}|\int_{\mathbb{R}^n}|a|^{n/2}|\widehat{f}(ar_{\theta}^{-1}\omega)|\int_{S0(n)}\int_{A_2^{'}a}^{\infty} |\widehat{\psi_{1}}(a_1{r_\theta}_1^{-1}(\omega))\ \widehat{\psi_2}(a_1 {r_{\theta}}_1^{-1}(\omega))|\ \frac{da_1d\theta_1}{a_1}d\omega\\
&\leq |C_{\psi_1\psi_2}^{-1}|\int_{\mathbb{R}^n}|p|^{-n/2}|\widehat{f}(\omega)|\int_{S0(n)}\int_{0}^{A_1^{'}} |\widehat{\psi_{1}}(a_1 {r_{\theta}}_1^{-1}(\omega))\ \widehat{\psi_2}(a_1 {r_{\theta}}_1^{-1}(\omega))|\ \frac{da_1d\theta_1}{a_1}d\omega \\
&\qquad+ |C_{\psi_1\psi_2}^{-1}|\int_{\mathbb{R}^n}|p|^{-n/2}|\widehat{f}(\omega)|\int_{S0(n)}\int_{A_2^{'}}^{\infty} |\widehat{\psi_{1}}(a_1{r_\theta}_1^{-1}(\omega))\ \widehat{\psi_2}(a_1 {r_{\theta}}_1^{-1}(\omega))|\ \frac{da_1d\theta_1}{a_1}d\omega.  
\end{align*}
Since both $\psi_1$ and $\psi_2$ are admissible, therefore we get
\begin{align}\label{eq4}
\int_{S0(n)}\int_{0}^{A_1^{'}} |\widehat{\psi_{1}}(a_1{r_\theta}_1^{-1}(\omega))\ \widehat{\psi_2}(a_1{r_\theta}_1^{-1}(\omega))|\frac{da_1d\theta_1}{a_1}\leq C_{\psi_1}C_{\psi_2}< \infty
\end{align}
 and
\begin{align}\label{eq2}
\int_{S0(n)}\int_{A_2^{'}}^{\infty} |\widehat{\psi_{1}}(a_1{r_\theta}_1^{-1}(\omega))\ \widehat{\psi_2}(a_1{r_\theta}_1^{-1}(\omega))|\frac{da_1d\theta_1}{a_1}\leq C_{\psi_1}C_{\psi_2}< \infty.
\end{align}
Also,
\begin{align*}
\lim\limits_{A_1^{'}\to 0}\int_{S0(n)}\int_{0}^{A_1^{'}}|\widehat{\psi_{1}}(a_1{r_\theta}_1^{-1}(\omega))\ \widehat{\psi_2}(a_1{r_\theta}_1^{-1}(\omega))|\frac{da_1d\theta_1}{a_1}  =0
\end{align*}
and
\begin{align*}
&\lim\limits_{A_2^{'}\to \infty}\  \int_{S0(n)}\int_{0}^{A_2^{'}} |\widehat{\psi_{1}}(a_1{r_\theta}_1^{-1}(\omega))\ \widehat{\psi_2}(a_1{r_\theta}_1^{-1}(\omega))|\frac{da_1d\theta_1}{a_1} & \\
 &\qquad \qquad =\int_{S0(n)}\int_{0}^{\infty} |\widehat{\psi_{1}}(a_1{r_\theta}_1^{-1}(\omega))\ \widehat{\psi_2}(a_1{r_\theta}_1^{-1}(\omega))|\frac{da_1d\theta_1}{a_1}.
\end{align*}
Thus,
\begin{align*}
\lim\limits_{A_2^{'}\to \infty} \int_{S0(n)}\int_{A_2^{'}}^{\infty}|\widehat{\psi_{1}}(a_1{r_\theta}_1^{-1}(\omega))\ \widehat{\psi_2}(a_1{r_\theta}_1^{-1}(\omega))|\frac{da_1d\theta_1}{a_1}= 0
\end{align*}
By Lebesgue dominated convergence theorem, we get
\begin{align*}
\lim\limits_{A_1^{'} \to 0} \int_{\mathbb{R}^n}|\widehat{f}(\omega)|\int_{S0(n)}\int_{0}^{A_1^{'}} |\widehat{\psi_{1}}(a_1{r_\theta}_1^{-1}(\omega))\ \widehat{\psi_2}(a_1 {r_{\theta}}_1^{-1}(\omega))|\ \frac{da_1d\theta_1}{a_1}d\omega = 0. 
\end{align*}
and
\begin{align*}
\lim\limits_{A_2^{'} \to \infty} \int_{\mathbb{R}^n}|\widehat{f}(\omega)|\int_{S0(n)}\int_{A_2^{'}}^{\infty} |\widehat{\psi_{1}}(a_1{r_\theta}_1^{-1}(\omega))\ \widehat{\psi_2}(a_1 {r_{\theta}}_1^{-1}(\omega))|\ \frac{da_1d\theta_1}{a_1}d\omega = 0. 
\end{align*}
Therefore, we can choose  $0 < A_1\leq A_2$ such that for any $(a,r_{\theta},t)$ with $0< p\leq a$ and  $0< A_1^{'}\leq A_1, A_2<A_2^{'}$, the conclusion of the theorem holds.
\end{proof}
\section{Wavelet group of the form $B\ltimes A$}
\noindent  Let $A$ and $B$ be  locally compact, second countable, unimodular, group of type I and $\tau$ be a map from $B$ to $\text{Aut(A)},$ where $\text{Aut}(A)$ is a group of automorphism of $A$. Let $G = B\ltimes_{\tau} A$ with  group operation $(b,a)(b_1,a_1) = (bb_1, a\tau_{b}(a_1))$. The left Haar measure $\mu$ on $G$ is given by $d\mu(b,a)=  \delta_{\tau}(b)\ da\ db$, where $\delta_{\tau}$ is a positive homomorphism of $B$ satisfying  $d(a) = \delta_{\tau}(b)\ d(\tau_{b}(a)).$ Define a representation $\pi$ of $G$ on  Hilbert space $L^2(A)$ such that
$$\pi(b,a)f(x)= f_{ba}(x)= \delta_{\tau}(b)^{1/2}f(\tau_{b^{-1}}(a^{-1}x)).$$
A function $\psi \in L^1\cap L^2(A)$ is said to be a feasible wavelet, if there exists a constant $0< C_{\psi} < \infty$ such that for almost every $\gamma \in \widehat{A}$ and for  all $f,g \in L^1\cap L^2(A)$
\begin{flalign*}
&\int_{B} \text{tr}(\gamma(f)\gamma(g)\gamma(\psi \circ \tau_{b^{-1}})^\star\gamma(\psi \circ \tau_{b^{-1}} ))\ \delta_{\tau}^2(b)\ d(b) = C_{\psi}\ \text{tr}(\gamma(f)\gamma(g)).
\end{flalign*} 
It may be noted that if $A$ is a locally compact abelian group, then for a  feasible wavelet $\psi$, 
$$C_{\psi}= \int_{B}|\widehat{\psi}(\gamma \circ \tau_{b})|^2 db,$$
is independent of almost every $\gamma\in \widehat{A}. $
For the case when $A$ is a locally compact abelian group, one may refer to \cite{fas;03}.
The continuous wavelet transform with respect to a feasible wavelet $\psi$ is an operator $W_{\psi}:L^2(A)\to L^2(G)$ defined as
$W_{\psi}f(b,a)= \langle f,\pi(b,a)\psi\rangle.$
 Also,
$$W_{\psi}f(b,a)=\int_{A}f(x)\overline{ \delta_{\tau}(b)^{1/2}\psi(\tau_{b^{-1}}(a^{-1}x))}dx = \delta_{\tau}(b)^{1/2}f*(\psi\circ\tau_{b^{-1}})^{\star}(a),$$
where $(\psi\circ\tau_{b^{-1}})^{\star}(x)=\overline{(\psi\circ\tau_{b^{-1}})(x^{-1})}$.   
Using feasibility condition and  Plancherel's  theorem \cite[Theorem 7.36]{fol} for locally compact, unimodular, type I groups, one can show the following result:
\begin{lem}\label{isom}
Let $\psi$ be a feasible wavelet and $f,g\in L^2(A)$. Then 
$\langle W_{\psi}f,W_{\psi}g\rangle = C_{\psi}\langle f,g \rangle$.
\end{lem}
\noindent In fact, one can show that $W_{\psi}(L^2(A))$ is a reproducing kernel Hilbert space with point-wise bounded kernel $k$ given by 
$$k((b,a),(b',a'))= \frac{1}{C_{\psi}}\langle \psi_{b'a'}, \psi_{ba}\rangle. $$ 
 Proceeding as in Lemma \ref{conn} and Theorem \ref{concth}, we obtain the following: 
\begin{thm}
Let  $G = B\ltimes_{\tau} A$ be an abstract wavelet group with $A_0$ non compact and $\psi \in L^2(A)$ be a feasible wavelet. If $0\neq f \in L^2(A)$, then  the set $\{ (b,a):W_{\psi}f(b,a) \neq 0\} $ has infinite measure.
\end{thm}
\begin{thm}
Let $G = B\ltimes_{\tau} A$ be an abstract wavelet  group and $\psi \in L^2(A)$ be a feasible wavelet. Let $M\subset G$ be such that $\mu(M)\leq \frac{\sqrt{C_{\psi}}}{\|\psi\|}.$
Then for any $f \in L^2(A)$,
\begin{align*}
\| W_{\psi}f - \chi_{M}W_{\psi}f\|\geq \sqrt{C_{\psi}}\left(1 - \frac{\|\psi\|\ {\mu(M)}^{1/2}}{\sqrt{C_{\psi}}}\right)\|f\|.
\end{align*}
\end{thm}
\begin{proof}
Let $P_R :L^2(G) \to L^2(G) $ be the orthogonal projection from $L^2(G)$ to the closed subspace $W_{\psi}(L^2(A))$ and $P_M :L^2(G) \to L^2(G)$ be the orthogonal projection from $L^2(G)$ to the closed subspace of functions with support contained in $M.$ Now, 
\begin{align}\label{conc}
\| W_{\psi}f - \chi_{M}\cdot W_{\psi}f\| &= \| W_{\psi}f- P_MP_R(W_{\psi}f)\| \nonumber \\ \nonumber
& \geq \|W_{\psi}f\| - \|P_MP_R\|\ \|W_{\psi}f\|\\
&= (1-\|P_MP_R\|)\ \sqrt{C_{\psi}}\ \|f\|.
\end{align}  
Now $P_R$ being a projection on a reproducing kernel Hilbert space can be represented by
\begin{align*}
P_RF(b,a) = \langle F(b^{'},a^{'}), k((b^{'},a^{'}), (b,a)) \rangle, 
\end{align*}
where $F \in L^2(G)$ and $k(.,.)$ is the reproducing kernel. Thus, for every $F\in L^2(G),$ we have
\begin{align*}
P_MP_RF(b,a) = 
\int_{G} \chi_{M}(b,a)k((b^{'},a^{'}), (b,a))F(b^{'},a^{'})d\mu(b^{'},a^{'}).
\end{align*} 
 From \cite{hal;78}, the operator norm $\|P_MP_R\|$  is given by
\begin{flalign*}
\| P_MP_R\|^{2} &= \int_{G}\int_{G}|\chi_{M}(b,a)\ k((b^{'},a^{'}),(b,a))|^2\ d\mu(b,a)\ d\mu(b^{'},a^{'})&\\
&=\int_{G}\int_{G}|\chi_{M}(b,a)\frac{1}{C_{\psi}}\langle \psi_{ba}, \psi_{b^{'}a^{'}}\rangle |^2\ d\mu(b,a)\ d\mu(b^{'},a^{'})&\\
&=\frac{1}{C_{\psi}^2}\int_{M} \int_{G}|W_{\psi}\psi_{ba}(b^{'},a^{'})|^2\ d\mu(b^{'},a^{'})\ d\mu(b,a)&\\
&=\frac{1}{C_{\psi}}\int_{M}\|\psi_{ba}\|^2\ d\mu(b,a)&\\
&= \frac{1}{C_{\psi}}\|\psi\|^{2}\ \mu(M).
\end{flalign*}
On substituting the value $\|P_MP_R\|$ into (\ref{conc}), we get the desired inequality.
\end{proof}
\noindent Next, we prove an analogue of Heisenberg type inequality for wavelet transform.
\begin{thm}
Let $G = B \ltimes_{\tau}(L\times \mathbb{R})$ be a wavelet group, where $L$ is a compact group and $\psi$ be a feasible wavelet. Then for any $f\in L^2(\mathbb{R})$, we have 
\begin{align*}
&\left(\int_{B}\int_{L}\int_{\mathbb{R}}|t|^2|W_{\psi}f(b,l,t)|^2\delta_{\tau}(b) dt\ dl\ db \right)^{1/2}
 \left( \int_{\widehat{L}}\int_{\widehat{R}}|\gamma|^2\|(\delta,\gamma)f\|^2_{\text{HS}}\ d\delta\ d\gamma\right)^{1/2}\\
&\geq\frac{\sqrt{C_{\psi}}}{2}||f||^2.
\end{align*}
\end{thm}
\begin{proof}
For almost all $b\in B$, using Heisenberg inequality for Fourier transform on $L\times \mathbb{R}$ (see \cite{bans;15}), we have
\begin{align*}
&\Big(\int_{L}\int_{\mathbb{R}}|t|^2|W_{\psi}f(b,l,t)|^2dl\  dt\Big)^{1/2}\Big(\int_{\widehat{L}}\int_{\mathbb{R}}|\gamma|^2\|(\delta, \gamma)W_{\psi}f(b,\cdot,\cdot)\|_{\text{HS}}^2\ d\delta\ d\gamma \Big)^{1/2}\\
& \geq \frac{1}{2}\int_{L}\int_{\mathbb{R}}|W_{\psi}f(b,l,t)|^2dl\ dt
\end{align*}
On integrating both sides with respect to the measure $\delta_{\tau}(b)db$, we get
\begin{align*}
&\int_B\Big(\int_{L}\int_{\mathbb{R}}|t|^2|W_{\psi}f(b,l,t)|^2 dl\ dt\Big)^{1/2}\Big(\int_{\widehat{L}}\int_{\mathbb{R}}|\gamma|^2\|(\delta, \gamma)W_{\psi}f(b,\cdot,\cdot)\|_{\text{HS}}^2 \ d\delta\ d\gamma \Big)^{1/2} \delta_{\tau}(b)db \\
&\geq  \frac{1}{2}\int_{B}\int_{L}\int_{\mathbb{R}}|W_{\psi}f(b,l,t)|^2 \delta_{\tau}(b) dl\ dt\ db.
\end{align*}
Using Lemma \ref{isom} and  Cauchy-Schwarz inequality, it follows that
\begin{flalign}\label{cauc}
&\left(\int_B\int_{L}\int_{\mathbb{R}}|t|^2|W_{\psi}f(b,l,t)|^2 \delta_{\tau}(b) dl\ dt\ db\right)^{1/2}\nonumber &\\
& \times \left(\int_B\int_{\widehat{L}}\int_{\mathbb{R}}|\gamma|^2\|(\delta, \gamma)W_{\psi}f(b,\cdot,\cdot)\|_{\text{HS}}^2\ \delta_{\tau}(b) d\delta\ d\gamma\ db \right)^{1/2} \geq\frac{1}{2}C_{\psi}\|f\|^2.
 \end{flalign}
 Since, \begin{align*}
 \|(\delta,\gamma)W_{\psi}f(b,.,.)\|_{\text{HS}}^2 & = \text{tr}((\delta,\gamma)W_{\psi}f(b,.,.)(\delta,\gamma)W_{\psi}f(b,.,.)^*)\\
 &= \delta_{\tau}(b)\text{tr}((\delta,\gamma)f^*(\delta,\gamma)f(\delta,\gamma)(\psi \circ \tau_{b^{-1}})^*(\delta,\gamma)(\psi \circ \tau_{b^{-1}})),
 \end{align*}
 therefore, on using feasibility condition, we obtain
 \begin{align}\label{adm}
& \int_B\int_{\widehat{L}}\int_{\mathbb{R}}|\gamma|^2\|(\delta, \gamma)W_{\psi}f(b,\cdot,\cdot)\|_{\text{HS}}^2\  \delta_{\tau}(b) d\delta\ d\gamma\ db\nonumber\\
& = \int_{\widehat{L}}\int_{\mathbb{R}}|\gamma|^2\int_B \text{tr}((\delta,\gamma)f^*(\delta,\gamma)f(\delta,\gamma)(\psi \circ \tau_{b^{-1}})^*(\delta,\gamma)(\psi \circ \tau_{b^{-1}}))\delta_{\tau}(b)^2 db\ d\delta\ d\gamma\nonumber \\
 &=\int_{\widehat{L}}\int_{\mathbb{R}}|\gamma|^2C_{\psi}\text{tr}((\delta,\gamma)f^*(\delta,\gamma)f(\delta,\gamma))d\delta\ d\gamma \nonumber \\
 &= C_{\psi}\int_{\widehat{L}}\int_{\mathbb{R}}|\gamma|^2\|(\delta,\gamma)f\|_{\text{HS}}^2\ d\delta\ d\gamma.
 \end{align}
 On substituting (\ref{adm}) into (\ref{cauc}), we get 
 \begin{align*}
 &\left(\int_B\int_{L}\int_{\mathbb{R}}|t|^2|W_{\psi}f(b,l,t)|^2 \delta_{\tau}(b) dl\ dt\ db\right)^{1/2}
\left( \int_{\widehat{L}}\int_{\mathbb{R}}|\gamma|^2\|(\delta,\gamma)f\|_{\text{HS}}^2\ d\delta\ d\gamma\right)^{1/2} &\\
& \geq\frac{1}{2}\sqrt{C_{\psi}}\ \|f\|^2. &\qedhere 
\end{align*}
\end{proof}
\begin{rem}
\begin{enumerate}
\item[(i)] Using the data given in \cite{bans;15}, one may write the explicit form of the Heisenberg type inequality when  $A$ is a locally compact abelian group with non-compact identity component  or $n$-dimensional nilpotent Lie group or of the form $K\ltimes \mathbb{R},$ where $K$ is compact subgroup of  group of automorphism of $\mathbb{R}$.

\item[(ii)] As a particular case of above results, one can obtain results for abstract shearlet transform (see \cite{kam:15}) by taking $B= H\ltimes_{\lambda}K$, where $H$ and K are locally compact groups and $\lambda:H\to $ Aut$(K)$ is a homomorphism.
\end{enumerate}
\end{rem}
\section*{Acknowledgements}
\noindent 
The first author is supported by  UGC under joint UGC-CSIR  Junior Research Fellowship (Ref. No:21/12/2014(ii)EU-V).

\bibliographystyle{amsplain}

\end{document}